\documentclass[10pt,reqno]{amsart}

\usepackage{amsmath,amssymb}
\usepackage{graphicx}
\usepackage{algorithm}
\usepackage{enumerate}
\usepackage{tikz}
\usetikzlibrary{decorations.pathmorphing}
\usetikzlibrary{decorations.markings}
\usetikzlibrary{positioning}
\usetikzlibrary{shapes}

\usepackage{tabu}
\usepackage{url}

\usepackage{mathrsfs}
\usepackage{caption} 
\usepackage{subcaption}

\usetikzlibrary{matrix, calc, arrows}
\usetikzlibrary{graphs}
\usetikzlibrary{graphs.standard}
\usepackage{tkz-berge}

\usepackage{todonotes}

\newtheorem{theorem}{Theorem}[section]
\newtheorem{corollary}[theorem]{Corollary}
\newtheorem{proposition}[theorem]{Proposition}

\newtheorem{observation}{Observation}

\theoremstyle{definition}
\newtheorem{example}[theorem]{Example}
\newtheorem{definition}[theorem]{Definition}

\newtheorem{openquestion}{Open Question}

\newcommand{\rmv}[1]{}

\title{Partial Domination in Graphs}
\author[Case, Hedetniemi, Laskar, Lipman]{Benjamin M. Case \and Stephen T. Hedetniemi \and Renu C. Laskar \and Drew J. Lipman }
\address{Dept of Mathematical Sciences, Clemson University}

\begin{document}

	\begin{abstract}
	A set $S\subseteq V$ is a \textit{dominating set} of $G$ if every vertex in $V - S$ is adjacent to at least one vertex in $S$. The \textit{domination number} $\gamma(G)$ of $G$ equals the minimum cardinality of a dominating set $S$ in $G$; we say that such a set $S$ is a $\gamma$\textit{-set}. The single greatest focus of research in domination theory is the determination of the value of $\gamma(G)$. By definition, all vertices must be dominated by a $\gamma$-set. In this paper we propose relaxing this requirement, by seeking sets of vertices that dominate a prescribed fraction of the vertices of a graph. We focus particular attention on $1/2$ domination, that is, sets of vertices that dominate at least half of the vertices of a graph $G$. 
	\medskip
	
	Keywords:  partial domination, dominating set, partial domination number, domination number	
	
%
%
	\end{abstract}
	\maketitle
	\section{Introduction}
	
	Let $G=(V,E)$ be a graph with vertex set $V=\{v_1,v_2,...,v_n\}$ and \emph{order} $n = |V|$. The {\em open neighborhood} of a vertex $v$ is the set $N(v) := \{u\: |\: uv \in E\} $ of vertices $u$ that are adjacent to $v$; the \emph{closed neighborhood} of $v$, $N[v]:=N(v)\cup \{v\}.$ A set $S\subseteq V$ is a \emph{dominating set} of $G$ if every vertex in $V - S$ is adjacent to at least one vertex in $S$, or equivalently, if $N[S] := \bigcup_{u\in S} N[u] = V$. The \emph{domination number} $\gamma (G)$ of $G$ equals the minimum cardinality of a dominating set $S$ in $G$; we say that such a set $S$ is a \emph{$\gamma$-set}. 
	
	The overwhelming focus of the more than 3,000 papers that have been published on dominating sets in graphs has been on determining the properties of a wide variety of variations of dominating sets in graphs, good bounds for various domination numbers, and the complexity of computing domination numbers. The definitions of different types of dominating sets all have in common, however, that a set $S$ must satisfy $N[S] = V$ in order to be called a dominating set. This requirement is important in a wide variety of applications where one must provide some level of service or resource to \textit{every} member of a network. For example, if surveillance of every node in a network must be provided, this can be done by surveillance cameras located at the nodes in a dominating set. 
	
	Alternatively, it may not be necessary, or profitable, to provide complete coverage of a network. For example, providing some service to outlying areas may not be sufficiently profitable if the number of dominated nodes per vertex is low. In these cases, a company only seeks to dominate nodes in a network that are profitable to do so. This gives rise to the notion of \emph{partial domination} in graphs. 
	
	\begin{definition}\label{def:PDom}
		For any graph $G=(V,E)$ and proportion $p \in [0,1]$, a set $S\subseteq V$ is a \textit{p-dominating set} if \[\frac{|N[S]|}{|V|} \geq p.\]  The \textit{p-domination number} $\gamma_p(G)$ equals the minimum cardinality of a $p$-dominating set in $G$. 
	\end{definition}
	
	
	For example, we say that a set $S\subseteq V$ is a 1/2-dominating set if $\frac{|N[S]|}{|V|} \geq 1/2.$ The {1/2-domination number} $\gamma_{1/2}(G)$ equals the minimum cardinality of a 1/2-dominating set in $G$. 
	
	We point out that a $\gamma_p$-set is not in general related to a $\gamma$-set. In particular a $\gamma$-set does not necessarily contain a $\gamma_{p}$-set. Equivalently, a $\gamma_{p}$-set cannot necessarily be extended to $\gamma$-set. To see this consider the graph in Figure \ref{ex:Disjoint} where the $\gamma$-set denoted by triangles is disjoint from $\gamma_{1/2}$-set consisting of just the square vertex.

	\begin{figure}[hh] 
		\centering
		\begin{tikzpicture}
		\node [draw,rectangle] (A) at (0,0) {};
		\node [draw,regular polygon, regular polygon sides=3,inner sep=1.5pt] (B) at (1,0) {};
		\node [draw,regular polygon, regular polygon sides=3,inner sep=1.5pt] (C) at (.707,.707) {};
		\node [draw,regular polygon, regular polygon sides=3,inner sep=1.5pt] (D) at (0,1) {};
		\node [draw,regular polygon, regular polygon sides=3,inner sep=1.5pt] (E) at (-.707,.707) {};
		\node [draw,regular polygon, regular polygon sides=3,inner sep=1.5pt] (F) at (-1,0) {};
		\node [draw,regular polygon, regular polygon sides=3,inner sep=1.5pt] (G) at (-.707,-.707) {};
		\node [draw,regular polygon, regular polygon sides=3,inner sep=1.5pt] (H) at (0,-1) {};
		\node [draw,regular polygon, regular polygon sides=3,inner sep=1.5pt] (I) at (.707,-.707) {};
		\node [draw,circle] (J) at (2,0) {};
		\node [draw,circle] (K) at (1.414,1.414) {};
		\node [draw,circle] (L) at (0,2) {};
		\node [draw,circle] (M) at (-1.414,1.414) {};
		\node [draw,circle] (N) at (-2,0) {};
		\node [draw,circle] (O) at (-1.414,-1.414) {};
		\node [draw,circle] (P) at (0,-2) {};
		\node [draw,circle] (Q) at (1.414,-1.414) {};

		\draw[line width=1pt] (A) edge (B);
		\draw[line width=1pt] (B) edge (J);
		\draw[line width=1pt] (A) edge (C);
		\draw[line width=1pt] (C) edge (K);
		\draw[line width=1pt] (A) edge (D);
		\draw[line width=1pt] (D) edge (L);
		\draw[line width=1pt] (A) edge (E);
		\draw[line width=1pt] (E) edge (M);
		\draw[line width=1pt] (A) edge (F);
		\draw[line width=1pt] (F) edge (N);
		\draw[line width=1pt] (A) edge (G);
		\draw[line width=1pt] (G) edge (O);
		\draw[line width=1pt] (A) edge (H);
		\draw[line width=1pt] (H) edge (P);
		\draw[line width=1pt] (A) edge (I);
		\draw[line width=1pt] (I) edge (Q);
		\end{tikzpicture}
		\caption{The $\gamma$-set denoted by triangles is disjoint from the $\gamma_{1/2}$-set consisting of just the square vertex. \label{ex:Disjoint}}
	\end{figure}
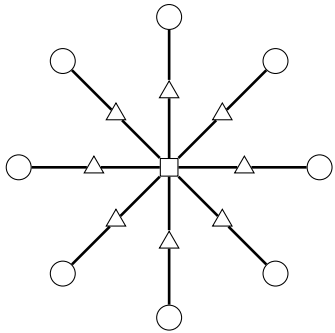

	We note that the $p$-domination number should not be confused with the well-studied \emph{fractional domination number}, denoted $\gamma_f(G)$, that is defined as follows. Let $f: V \to [0,1]$ be a function which assigns to each vertex $v\in V$ a rational number in the unit interval $[0,1]$. A function $f$ is called a \emph{dominating function} if for every $v\in V$, $f(N[v])\geq 1$, that is, the sum of the values $f(w)$ for every $w\in N[v]$ is greater than or equal to 1. The \emph{weight} of a fractional dominating function is simply the sum of all values $f(v)$ for every vertex $v\in V$. The \emph{fractional domination number} $\gamma_f(G)$ equals the minimum weight of a fractional dominating function on $G$. Fractional domination was introduced by Hedetniemi et al. in 1987 \cite{Hed1987} and has received considerable study since then.  The reader is referred to a chapter on fractional domination by Domke et al. in \cite{Domke} and the PhD thesis on fractional domination by Rubalcaba in 2005 \cite{Rubalcaba}.
	
	The partial domination number $\gamma_{p}(G)$ should also not be confused with the $\alpha$\emph{-domination number} $\gamma_\alpha(G),$ for a given value $\alpha \in [0,1]$, which is defined as the minimum cardinality of a set $S$ having the property that for every vertex $v\in V$, $|N[v]\cap S|/N[v]\geq \alpha$, that is, the set $S$ dominates at least the fraction $\alpha$ of the vertices in every closed neighborhood $N[v]$. Alpha domination was introduced by  Dunbar et al. in 2000 \cite{Dunbar}.  The interested reader is referred to a recent paper on alpha domination by Jafari Rad and Volkmann in 2016 \cite{Rad}.

\section{Examples and $p$-Domination for Classes of Graphs}
\label{sec:Classes}

	We begin our study of $p$-domination in graphs by considering some motivating examples and classes of graphs. In particular we will determine $\gamma_p$ for complete multipartite graphs, cycles, paths, grid graphs, and cylinders. To begin we give the following simple example. 
	
	\begin{example}\label{ex:Path}
		Consider the path on six vertices, a $1/2$-dominating set is given by taking any vertex that is not a leaf, see Figure \ref{fig:Path}.  We point out again using this example that each vertex in a $p$-dominating set dominates itself and its neighbors. 
	\end{example}	
	
		\begin{figure}[hh]
		\centering
			\begin{tikzpicture}[every loop/.style={}]
			\draw (0,1) circle (.15) node (A) {};
			\filldraw (1,1) circle (.15) node (B) {};
			\draw (2,1) circle (.15) node (C) {}
			(3,1) circle (.15) node (D) {}
			(4,1) circle (.15) node (E) {}
			(5,1) circle (.15) node (F)  {};
			\draw[line width=1pt] (A) edge (B);
			\draw[line width=1pt] (B) edge (C);
			\draw[line width=1pt] (C) edge (D);
			\draw[line width=1pt] (D) edge (E);
			\draw[line width=1pt] (E) edge (F);
			\end{tikzpicture}
		\caption{The path on six vertices, the shaded vertex gives a 1/2-dominating set. \label{fig:Path}}
		\end{figure}
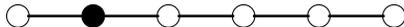
	
	We now consider complete bipartite graphs. 
	\begin{proposition}\label{prop:CompleteBi}
	In any complete bipartite graph $K_{m,n}$ one can find a $1/2$-dominating set by simply choosing one vertex from the side with fewer vertices.  
	\end{proposition}
	
	\begin{example}\label{ex:Complete3,5}
		For example consider $K_{3,5}$ show in Figure \ref{fig:Complete3,5}. Let $S=\{v\}$ a vertex from the part with three vertices, then $\frac{|N[v]|}{|V|}=\frac{6}{8}$. 
		\end{example}
	
		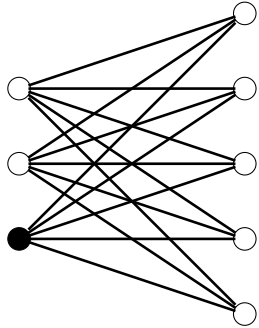
\begin{figure}[hh]
		\centering
			\begin{tikzpicture}[every loop/.style={}]
			\filldraw
			(0,2) circle (.15) node (A) {};
			\draw
			(0,3) circle (.15) node (B) {}
			(0,4) circle (.15) node (C) {}
			
			(3,1) circle (.15) node (D) {}
			(3,2) circle (.15) node (E) {}
			(3,3) circle (.15) node (F) {}
			(3,4) circle (.15) node (G) {}
			(3,5) circle (.15) node (H) {};
			
			\draw[line width=1pt] (A) edge (D);
			\draw[line width=1pt] (A) edge (E);
			\draw[line width=1pt] (A) edge (F);
			\draw[line width=1pt] (A) edge (G);
			\draw[line width=1pt] (A) edge (H);
			
			\draw[line width=1pt] (B) edge (D);
			\draw[line width=1pt] (B) edge (E);
			\draw[line width=1pt] (B) edge (F);
			\draw[line width=1pt] (B) edge (G);
			\draw[line width=1pt] (B) edge (H);
			
			\draw[line width=1pt] (C) edge (D);
			\draw[line width=1pt] (C) edge (E);
			\draw[line width=1pt] (C) edge (F);
			\draw[line width=1pt] (C) edge (G);
			\draw[line width=1pt] (C) edge (H);	
			\end{tikzpicture}
		\caption{The graph $K_{3,5}$ with a one vertex 1/2-dominating set, indicated by the shaded vertex. \label{fig:Complete3,5}}
		\end{figure}
We can generalize this to any complete multipartite graph. 	
\begin{proposition}\label{prop:CompleteMulti}
	In any complete multipartite graph $K_{m_1,m_2,...,m_k}$ one can find a $1/2$-dominating set by simply choosing one vertex from the independent set with the fewest vertices.  
\end{proposition}
\begin{example}\label{ex:Cycle}
Consider the cycle on twelve vertices. Any choice of two vertices, with disjoint closed neighborhoods, will form a 1/2-dominating set, see Figure \ref{fig:Cycle}.
\end{example}
	
	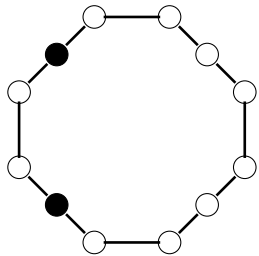
\begin{figure}
	\centering
		\begin{tikzpicture}[every loop/.style={}]
		
		\draw
		(1,0) circle (.15) node (A) {};
		\filldraw
		(.5,.5) circle (.15) node (B) {};
		\draw
		(0,1) circle (.15) node (C) {};
		\draw 
		(0,2) circle (.15) node (D) {};
		\filldraw	
		(.5,2.5) circle (.15) node (E) {};
		\draw
		(1,3) circle (.15) node (F) {}
		(2,3) circle (.15) node (G) {}
		(2.5,2.5) circle (.15) node (H) {}
		(3,2) circle (.15) node (I) {}
		(3,1) circle (.15) node (J) {}
		(2.5,.5) circle (.15) node (K) {}
		(2,0) circle (.15) node (L) {};
		
		\draw[line width=1pt] (A) edge (B);
		\draw[line width=1pt] (B) edge (C);
		\draw[line width=1pt] (C) edge (D);
		\draw[line width=1pt] (D) edge (E);
		\draw[line width=1pt] (E) edge (F);
		\draw[line width=1pt] (F) edge (G);
		\draw[line width=1pt] (G) edge (H);
		\draw[line width=1pt] (H) edge (I);
		\draw[line width=1pt] (I) edge (J);
		\draw[line width=1pt] (J) edge (K);
		\draw[line width=1pt] (K) edge (L);
		\draw[line width=1pt] (L) edge (A);
		\end{tikzpicture}
	\caption{The cycle on 12 vertices with a 1/2-dominating set, indicated by the shaded vertices. \label{fig:Cycle}}
\end{figure}

Generalizing this, for paths and cycles we can get precise statements about $\gamma_{1/2}(G)$. 

\begin{proposition}\label{prop:Cycles}
	For a cycle of length $n$, \[\gamma_{1/2}(C_n) = \lceil n/6 \rceil. \]
\end{proposition}
	
\begin{proof}
	Let $G$ be a cycle of length $n$. Choose a vertex anywhere on the cycle; it dominates three vertices: itself and its two neighbors. Now move three vertices in one direction around the cycle and add that vertex to $S$; the number of dominated vertices has gone up by three. Continue moving in the same direction around the cycle choosing every third vertex until you have gone half way around the cycle. Clearly, $S$ dominates at least 1/2 of the vertices of graph. It is minimal because each vertex dominated by $S$ is dominated by only one vertex in $S$. One out of every three on one half of the cycle is in $S$, so one out of every six vertices overall is in $S$. The ceiling is needed if $n$ is not divisible by $3$ to ensure that $S$ dominates at least $1/2$.  
\end{proof}

\begin{proposition}\label{prop:Paths}
	For any path of length $n$, \[\gamma_{1/2}(P_n) = \lceil n/6 \rceil.\]
\end{proposition}
\begin{proof} This follows from the proof of Proposition \ref{prop:Cycles} if we think of a path as a cycle with an edge deleted. More precisely, for a cycle of length $n$ there is a 1/2-dominating set that is not incident to some edge.  One such 1/2-dominating set was shown by construction in the poof of Proposition \ref{prop:Cycles}. This edge can be removed to make the cycle a path without affecting which nodes are dominated.  
\end{proof}

We now turn to considering grid graphs. It had for many years proven difficult to determine formulas for the domination numbers of all grid graphs, 16 different formulas were determined by Goncalves et al. in 2011 \cite{Goncalves}, and more work was done on constructing such $\gamma$-sets in \cite{Hutson}.  It is surprisingly simple, however, to determine the 1/2-domination
number of all grid graphs, as follows.

\begin{theorem}\label{thm:gridgraphs}
	For the $m$-by-$n$ grid graph $P_m \Box P_n$, $m\leq n$, the 1/2-domination
	number is as follows:
	\begin{enumerate}
		\item for $m = 1$, $\gamma_{1/2}(P_n) = \lceil n/6 \rceil$,
		\item for $m = 2$, $\gamma_{1/2}(P_2 \Box P_n) = \lceil n/4 \rceil$,
		\item for $m \geq 3$, $\gamma_{1/2}(P_m \Box P_n) = \lceil mn/10 \rceil$.
	\end{enumerate}

\end{theorem}

\begin{proof}
	Statement (1) follows from Proposition \ref{prop:Paths}.
	For (2) we can always construct a dominating set such that the closed neighborhoods of the points in the $\gamma_{1/2}$-set are disjoint and each vertex dominates four vertices, see Figure \ref{pf:case2}.
	For (3) we can always construct a dominating set such that the closed neighborhoods of the points in the $\gamma_{1/2}$-set are disjoint and each vertex dominates five vertices, see Figure \ref{pf:case3}.
\end{proof}

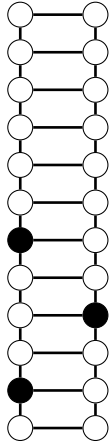
\begin{figure}[hh]
	\centering
	\begin{tikzpicture}
	\node [draw,circle] (A) at (0,0) {};
	\node [draw,circle,fill] (B) at (0,.5) {};
	\node [draw,circle] (C) at (0,1) {};
	\node [draw,circle] (D) at (0,1.5) {};
	\node [draw,circle] (E) at (0,2) {};
	\node [draw,circle,fill] (F) at (0,2.5) {};
	\node [draw,circle] (G) at (0,3) {};
	\node [draw,circle] (H) at (0,3.5) {};
	\node [draw,circle] (I) at (0,4) {};
	\node [draw,circle] (J) at (0,4.5) {};
	\node [draw,circle] (K) at (0,5) {};
	\node [draw,circle] (L) at (0,5.5) {};
	\node [draw,circle] (A2) at (1,0) {};
	\node [draw,circle] (B2) at (1,.5) {};
	\node [draw,circle] (C2) at (1,1) {};
	\node [draw,circle,fill] (D2) at (1,1.5) {};
	\node [draw,circle] (E2) at (1,2) {};
	\node [draw,circle] (F2) at (1,2.5) {};
	\node [draw,circle] (G2) at (1,3) {};
	\node [draw,circle] (H2) at (1,3.5) {};
	\node [draw,circle] (I2) at (1,4) {};
	\node [draw,circle] (J2) at (1,4.5) {};
	\node [draw,circle] (K2) at (1,5) {};
	\node [draw,circle] (L2) at (1,5.5) {};
	\draw[line width=1pt] (A) edge (A2);
	\draw[line width=1pt] (B) edge (B2);
	\draw[line width=1pt] (C) edge (C2);
	\draw[line width=1pt] (D) edge (D2);
	\draw[line width=1pt] (E) edge (E2);
	\draw[line width=1pt] (F) edge (F2);
	\draw[line width=1pt] (G) edge (G2);
	\draw[line width=1pt] (H) edge (H2);
	\draw[line width=1pt] (I) edge (I2);
	\draw[line width=1pt] (J) edge (J2);
	\draw[line width=1pt] (K) edge (K2);
	\draw[line width=1pt] (L) edge (L2);
	\draw[line width=1pt] (A) edge (B);
	\draw[line width=1pt] (B) edge (C);
	\draw[line width=1pt] (C) edge (D);
	\draw[line width=1pt] (D) edge (E);
	\draw[line width=1pt] (E) edge (F);
	\draw[line width=1pt] (F) edge (G);
	\draw[line width=1pt] (G) edge (H);
	\draw[line width=1pt] (H) edge (I);
	\draw[line width=1pt] (I) edge (J);
	\draw[line width=1pt] (J) edge (K);
	\draw[line width=1pt] (K) edge (L);
	\draw[line width=1pt] (A2) edge (B2);
	\draw[line width=1pt] (B2) edge (C2);
	\draw[line width=1pt] (C2) edge (D2);
	\draw[line width=1pt] (D2) edge (E2);
	\draw[line width=1pt] (E2) edge (F2);
	\draw[line width=1pt] (F2) edge (G2);
	\draw[line width=1pt] (G2) edge (H2);
	\draw[line width=1pt] (H2) edge (I2);
	\draw[line width=1pt] (I2) edge (J2);
	\draw[line width=1pt] (J2) edge (K2);
	\draw[line width=1pt] (K2) edge (L2);
	\end{tikzpicture}
	\caption{Theorem \ref{thm:gridgraphs} (2): $P_2\Box P_{12}$ dominated by 3 vertices with disjoint closed neighborhoods.  
	\label{pf:case2}}
\end{figure}

\begin{figure}[hh]
	\centering
	\begin{tikzpicture}
	\node [draw,circle] (A) at (0,0) {};
	\node [draw,circle] (B) at (0,.5) {};
	\node [draw,circle] (C) at (0,1) {};
	\node [draw,circle] (D) at (0,1.5) {};
	\node [draw,circle] (E) at (0,2) {};
	\node [draw,circle] (F) at (0,2.5) {};
	\node [draw,circle] (G) at (0,3) {};
	\node [draw,circle] (H) at (0,3.5) {};
	\node [draw,circle] (I) at (0,4) {};
	\node [draw,circle] (J) at (0,4.5) {};
	\node [draw,circle] (K) at (0,5) {};
	\node [draw,circle] (L) at (0,5.5) {};
	
	\node [draw,circle] (A2) at (1,0) {};
	\node [draw,circle,fill] (B2) at (1,.5) {};
	\node [draw,circle] (C2) at (1,1) {};
	\node [draw,circle] (D2) at (1,1.5) {};
	\node [draw,circle] (E2) at (1,2) {};
	\node [draw,circle,fill] (F2) at (1,2.5) {};
	\node [draw,circle] (G2) at (1,3) {};
	\node [draw,circle] (H2) at (1,3.5) {};
	\node [draw,circle] (I2) at (1,4) {};
	\node [draw,circle,fill] (J2) at (1,4.5) {};
	\node [draw,circle] (K2) at (1,5) {};
	\node [draw,circle] (L2) at (1,5.5) {};
	
	\node [draw,circle] (A3) at (2,0) {};
	\node [draw,circle] (B3) at (2,.5) {};
	\node [draw,circle] (C3) at (2,1) {};
	\node [draw,circle,fill] (D3) at (2,1.5) {};
	\node [draw,circle] (E3) at (2,2) {};
	\node [draw,circle] (F3) at (2,2.5) {};
	\node [draw,circle] (G3) at (2,3) {};
	\node [draw,circle,fill] (H3) at (2,3.5) {};
	\node [draw,circle] (I3) at (2,4) {};
	\node [draw,circle] (J3) at (2,4.5) {};
	\node [draw,circle] (K3) at (2,5) {};
	\node [draw,circle] (L3) at (2,5.5) {};
	
	\node [draw,circle] (A4) at (3,0) {};
	\node [draw,circle] (B4) at (3,.5) {};
	\node [draw,circle] (C4) at (3,1) {};
	\node [draw,circle] (D4) at (3,1.5) {};
	\node [draw,circle] (E4) at (3,2) {};
	\node [draw,circle] (F4) at (3,2.5) {};
	\node [draw,circle] (G4) at (3,3) {};
	\node [draw,circle] (H4) at (3,3.5) {};
	\node [draw,circle] (I4) at (3,4) {};
	\node [draw,circle] (J4) at (3,4.5) {};
	\node [draw,circle] (K4) at (3,5) {};
	\node [draw,circle] (L4) at (3,5.5) {};
	
	\draw[line width=1pt] (A) edge (A2);
	\draw[line width=1pt] (B) edge (B2);
	\draw[line width=1pt] (C) edge (C2);
	\draw[line width=1pt] (D) edge (D2);
	\draw[line width=1pt] (E) edge (E2);
	\draw[line width=1pt] (F) edge (F2);
	\draw[line width=1pt] (G) edge (G2);
	\draw[line width=1pt] (H) edge (H2);
	\draw[line width=1pt] (I) edge (I2);
	\draw[line width=1pt] (J) edge (J2);
	\draw[line width=1pt] (K) edge (K2);
	\draw[line width=1pt] (L) edge (L2);
	
	\draw[line width=1pt] (A2) edge (A3);
	\draw[line width=1pt] (B2) edge (B3);
	\draw[line width=1pt] (C2) edge (C3);
	\draw[line width=1pt] (D2) edge (D3);
	\draw[line width=1pt] (E2) edge (E3);
	\draw[line width=1pt] (F2) edge (F3);
	\draw[line width=1pt] (G2) edge (G3);
	\draw[line width=1pt] (H2) edge (H3);
	\draw[line width=1pt] (I2) edge (I3);
	\draw[line width=1pt] (J2) edge (J3);
	\draw[line width=1pt] (K2) edge (K3);
	\draw[line width=1pt] (L2) edge (L3);
	
	\draw[line width=1pt] (A3) edge (A4);
	\draw[line width=1pt] (B3) edge (B4);
	\draw[line width=1pt] (C3) edge (C4);
	\draw[line width=1pt] (D3) edge (D4);
	\draw[line width=1pt] (E3) edge (E4);
	\draw[line width=1pt] (F3) edge (F4);
	\draw[line width=1pt] (G3) edge (G4);
	\draw[line width=1pt] (H3) edge (H4);
	\draw[line width=1pt] (I3) edge (I4);
	\draw[line width=1pt] (J3) edge (J4);
	\draw[line width=1pt] (K3) edge (K4);
	\draw[line width=1pt] (L3) edge (L4);
	
	\draw[line width=1pt] (A) edge (B);
	\draw[line width=1pt] (B) edge (C);
	\draw[line width=1pt] (C) edge (D);
	\draw[line width=1pt] (D) edge (E);
	\draw[line width=1pt] (E) edge (F);
	\draw[line width=1pt] (F) edge (G);
	\draw[line width=1pt] (G) edge (H);
	\draw[line width=1pt] (H) edge (I);
	\draw[line width=1pt] (I) edge (J);
	\draw[line width=1pt] (J) edge (K);
	\draw[line width=1pt] (K) edge (L);
	
	\draw[line width=1pt] (A2) edge (B2);
	\draw[line width=1pt] (B2) edge (C2);
	\draw[line width=1pt] (C2) edge (D2);
	\draw[line width=1pt] (D2) edge (E2);
	\draw[line width=1pt] (E2) edge (F2);
	\draw[line width=1pt] (F2) edge (G2);
	\draw[line width=1pt] (G2) edge (H2);
	\draw[line width=1pt] (H2) edge (I2);
	\draw[line width=1pt] (I2) edge (J2);
	\draw[line width=1pt] (J2) edge (K2);
	\draw[line width=1pt] (K2) edge (L2);
	
	\draw[line width=1pt] (A3) edge (B3);
	\draw[line width=1pt] (B3) edge (C3);
	\draw[line width=1pt] (C3) edge (D3);
	\draw[line width=1pt] (D3) edge (E3);
	\draw[line width=1pt] (E3) edge (F3);
	\draw[line width=1pt] (F3) edge (G3);
	\draw[line width=1pt] (G3) edge (H3);
	\draw[line width=1pt] (H3) edge (I3);
	\draw[line width=1pt] (I3) edge (J3);
	\draw[line width=1pt] (J3) edge (K3);
	\draw[line width=1pt] (K3) edge (L3);
	
	\draw[line width=1pt] (A4) edge (B4);
	\draw[line width=1pt] (B4) edge (C4);
	\draw[line width=1pt] (C4) edge (D4);
	\draw[line width=1pt] (D4) edge (E4);
	\draw[line width=1pt] (E4) edge (F4);
	\draw[line width=1pt] (F4) edge (G4);
	\draw[line width=1pt] (G4) edge (H4);
	\draw[line width=1pt] (H4) edge (I4);
	\draw[line width=1pt] (I4) edge (J4);
	\draw[line width=1pt] (J4) edge (K4);
	\draw[line width=1pt] (K4) edge (L4);
	\end{tikzpicture} 
	
	\caption{Theorem \ref{thm:gridgraphs} (3): $P_4\Box P_{12}$ dominated by 5 vertices with disjoint closed neighborhoods.  
	\label{pf:case3}}
\end{figure}

For grid graphs it is interesting to compare the formulas for domination to those for $p$-domination in order to see
just what the savings are when one only needs to dominate half of the vertices.
From \cite{Goncalves} for $m,n\geq 16$ 
\[\gamma(P_m \Box P_n) =\left  \lfloor \frac{(m+2)(n+2)}{5} \right \rfloor -4. \]
If one considers the ratio 
\[\frac{\gamma_{1/2}(P_m \Box P_n)}{\gamma(P_m \Box P_n)} \]
as $m,n$ grow large, the ratio approaches 0.5, but for not too large $m,n$ one can dominate half of the vertices with fewer than half of a the vertices in a dominating set.  It also follows from Theorem \ref{thm:gridgraphs} that the same formulas for $\gamma_{1/2}$-sets hold for cylinders. 
\begin{corollary}
	For the $m$-by-$n$ cylinder graph $C_m \Box C_n$, $m\leq n$, the 1/2-domination
	number is as follows:
	\begin{enumerate}
		\item for $m = 1$, $\gamma_{1/2}(C_n) = \lceil n/6 \rceil$,
		\item for $m = 2$, $\gamma_{1/2}(C_2 \Box C_n) = \lceil n/4 \rceil$,
		\item for $m \geq 3$, $\gamma_{1/2}(C_m \Box C_n) = \lceil mn/10 \rceil$.
	\end{enumerate}
\end{corollary}

\section{Bounds on the $p$-domination number}\label{sec:Bounds}
	In this section we consider various bounds we can get on the $p$-domination number. 
	First we consider how $\gamma_p(G)$ and $\gamma_q(G)$ relate to each other for two different proportions $p$ and $q$. 
	
	\begin{proposition}\label{prop:p<q}
		Let $0 \leq p < q \leq 1$. Then \[\gamma_p(G) \leq \gamma_q(G).\]
	\end{proposition}
	\begin{proof}
	 The proof follows from the observation that every $q$-dominating set is a $p$-dominating set. Moreover, equality will hold if and only if the $\gamma_p$-set dominates a proportion $q$ of the vertices.
	\end{proof}
	
	Setting $q=1$ gives a relation between classical domination and partial domination:
	
	\begin{corollary}\label{cor:DomVsPDom}
		For $q=1$ partial domination is the same a classical domination, thus we have an upper bound on $\gamma_p(G)$ for all $p$: 
		\[\gamma_p(G) \leq \gamma(G).\]
	\end{corollary}
	
	Now we consider some more interesting bounds on $\gamma_p$ coming from the classical domination number, observing that if you only need to dominate half of the vertices of the graph, you will only need at most half of the number of vertices in a $\gamma$-set of $G$ rounded up. 
	\begin{theorem}\label{thm:gamma1/2}
		For any connected graph $G$, \[\gamma_{1/2}(G)\leq \lceil \gamma(G)/2 \rceil .\]
	\end{theorem}

	 Note that without the ceiling on the right side of the inequality, this fails for a complete graph $K_n$ since one node is needed in a $\gamma$-set and still one node in a $\gamma_{1/2}$-set. The proof of this follows from the following more general statement.
	\begin{theorem}\label{thm:gammai/j}
		For any connected graph $G$, $\gamma_{i/j}(G)\leq \lceil i/j\gamma(G) \rceil$.
	\end{theorem}

	\begin{proof}
		Given a $\gamma$-set $B=\{v_1,...,v_r\}$, partition $V$ into sets $S_1,...,S_r$ such that $S_i\subseteq N[v_i], v_i\in S_i$. Without lost of generality, $|S_1|\geq \cdots \geq |S_r|.$ Define $B' = \{ v_1,...,v_{\lceil {ir}/{j}  \rceil }\}.$
		
		Claim: \[\left| \bigcup_{k = 1}^{\lceil ir/j \rceil } S_k  \right| \geq i/j |V|.\]
		By construction \[\left| \bigcup_{k = 1}^{\lceil ir/j \rceil } S_k  \right| + \left| \bigcup_{k = \lceil ir/j \rceil +1}^{r } S_k  \right| = |V|. \]
		
		Since the average size of $S_k$, $k=1,...,\lceil ri/j \rceil, $ is at least the averages size of all $S_k$'s, the result follows because at worst $|S_k|=|S_\ell|$ for all $k \neq \ell$ and here $\left| \bigcup_{k = 1}^{\lceil ir/j \rceil } S_k  \right| = \lceil i/j |V| \rceil.$
	\end{proof}

%
%
%

	Next consider some Nordhaus-Gaddum type bounds on the $i/j$-partial domination number. 
	
	\begin{theorem}\label{thm:NordGaddum}
		If $G$ and $\bar{G}$ are connected, then 
		\[\gamma_{i/j}(G) + \gamma_{i/j}(\bar{G}) \leq \left \lceil \frac{i}{j}\left(\left \lfloor \frac{n}{2}\right \rfloor +2 \right) \right \rceil +1 \]
	\end{theorem}
	\begin{proof}
		By Theorem  \ref{thm:gammai/j} applied to $G$ and $\bar{G}$ we get that  $\gamma_{i/j}(G) \lceil \leq i/j \gamma(G) \rceil $, and $\gamma_{i/j}(\bar{G}) \leq \lceil i/j \gamma(\bar{G}) \rceil$.  Adding these two inequalities gives
		\[  \gamma_{i/j}(G) +\gamma_{i/j}(\bar{G}) \leq \lceil  i/j \gamma(G) \rceil + \lceil i/j \gamma(\bar{G}) \rceil \leq  \lceil  i/j \left( \gamma(G) +  \gamma(\bar{G}) \right) \rceil +1. \]
		In the right hand side of this inequality we see $\gamma(G) + \gamma(\bar{G})$ which is a well studied quantity from classical domination theory. In particular Bollob\'as and Cockayne and also Joseph and Arumugam have given the following upper bound for it \cite{Harary}
		\[\gamma(G)+ \gamma(\bar{G}) \leq \lfloor n/2\rfloor +2.\]
		Combining this with the inequality above gives the result
		\[  \gamma_{i/j}(G) +\gamma_{i/j}(\bar{G}) \leq \left  \lceil  \frac{i}{j}\left( \left \lfloor \frac{n}{2}\right \rfloor +2  \right) \right \rceil   +1. \]
	\end{proof}
	In particular for $p=1/2$ we get the following:
	\begin{corollary}\label{cor:NordGaddum}
	If $G$ and $\bar{G}$ are connected, then 
	\[\gamma_{1/2}(G) + \gamma_{1/2}(\bar{G}) \leq \left \lceil \frac{1}{2}\left \lfloor \frac{n}{2}\right \rfloor  \right \rceil +2. \]
\end{corollary}
	
	\section{Related Parameters}
	In classical domination theory one often considers the quantity $\Gamma(G)$, which is defined as the maximum size of a dominating set that is minimal -- minimal meaning that if any nonempty subset of vertices is removed, it will no longer have the property of being a dominating set. We generalize this concept to define $\Gamma_{p}(G)$. 
	\begin{definition}
		\[ \Gamma_{p}(G)  := \max \{ |S| : S \text{ dominates a proportion }p\text{ and is minimal} \}.\]
	\end{definition}
	 
	This quantity is related to $\gamma_p(G)$ in the following way: 
	
	\begin{observation}\label{obs:gammaGamma} For any graph $G$, 
		$\gamma_{p}(G) \leq \Gamma_{p}(G)$.
	\end{observation}

	We point out in the following example that $\gamma_{p}(G) $ could be strictly less than $\Gamma_{p}(G)$. 
	
\begin{example}\label{ex:PathGamma} On the path with six vertices there is a $\Gamma_{1/2}(P_6)$ set given by taking the two leaf vertices see Figure \ref{fig:PathGamma}; this set dominates four vertices and is minimal. However, as we saw in Proposition \ref{prop:Paths}, $\gamma_{1/2}(P_6) = 1$. 
\end{example}

		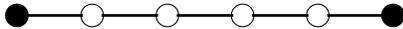
\begin{figure}[hh]
		\centering
			\begin{tikzpicture}[every loop/.style={}]
			\filldraw (0,1) circle (.15) node (A) {};
			\draw (1,1) circle (.15) node (B) {}
			(2,1) circle (.15) node (C) {}
			(3,1) circle (.15) node (D) {}
			(4,1) circle (.15) node (E) {};
			\filldraw
			(5,1) circle (.15) node (F)  {};
			\draw[line width=1pt] (A) edge (B);
			\draw[line width=1pt] (B) edge (C);
			\draw[line width=1pt] (C) edge (D);
			\draw[line width=1pt] (D) edge (E);
			\draw[line width=1pt] (E) edge (F);
			\end{tikzpicture}
		\caption{For the path on 6 vertices, the shaded leaf vertices define a 1/2-dominating set which is minimal but not minimum.\label{fig:PathGamma}}
\end{figure}
	For classical domination there is a well known chain of inequalities, first introduced in \cite{Cock1978}, that relates the domination number of a graph to several related parameters. We state the inequality without defining all the parameters here.
	\begin{proposition}\cite[Proposition 4.2]{Cock1978}  \label{prop:Domineq}
		For any graph $G$,
		\[ ir(G) \leq \gamma(G) \leq i(G) \leq \beta_0(G) \leq \Gamma(G) \leq IR(G).\]
	\end{proposition} 
	It would be an interesting furture direction to consider generalizing these concepts and results to partial domination. 
	\begin{openquestion}
		What meaning and relationships do the quantities in Proposition \ref{prop:Domineq} have in the setting of partial domination?
	\end{openquestion}

	\section{Computational Complexity}\label{sec:Complexity}
	We conclude our discussion of partial domination by considering the computational complexity of finding a partially dominating set. The problem of finding a dominating set is a classic NP-Hard problem \cite{Garey}, so it is an interesting question to ask how hard the problem is if we relax the requirement that all the vertices in the graph be dominated. In fact, we can consider the parameterized complexity of the problem in terms of the number of vertices that are dominated by a $\gamma_p$-set.  Before considering the question of the complexity of finding a partial dominating set, we need to consider the following related problem. 
	
	\begin{definition}
		The {\em t-Dominating Set Problem} is that of finding a set of at most $k$ nodes that dominate at least $t$ nodes of a graph $G=(V,E)$. 
	\end{definition}
	To understand the relationship of this problem to partial set domination, first note that $t-$domination is stated as considering a number of vertices in the graph rather than a proportion, but the two problems are easily related in this sense by considering $t = p \cdot |V|$ for a proportion $p$.  A more important difference to note is that the $t-$domination problem has no minimal conditions on the size of the set used to dominate; it is concerned with finding a set with at most $k$ nodes that dominates at least $t$ nodes.  The complexity of the $t-$Dominating Set Problem has been studied in \cite{Kneis}, where the authors prove the following result.
	
	\begin{theorem}[Kneis, et al. \cite{Kneis}] 
		t-Dominating set is fixed time parameterizable in parameter $t$ with a $O\left((4+\epsilon)^t poly(n)\right)$ randomized algorithm and a $O\left((16+\epsilon)^t poly(n)\right)$ deterministic algorithm. 
	\end{theorem}

 	Now consider the question of what one could do with an algorithm that solved the $t-$dominating set problem as a subroutine to solve partial set domination.  For any $k\in [0,n]$ one could use this subroutine to determine if there was a $t-$dominating set. Using this ability one could perform a binary search on $k\in [0,n]$ until finding the smallest $k$ having a $t-$dominating set. Such a $k$ would equal $\gamma_{p}(G)$ where $t = p \cdot |V|$. Since the complexity of a binary search adds a $\log n$ factor to the running time, we arrive at the following result that the partial domination problem is fixed time parameterizable.

 	\begin{theorem}
 		Partial dominating set is fixed time parameterizable in parameter $t:=p\cdot n$ with a $O\left( (4+\epsilon)^t poly(n)\log n\right)$ randomized algorithm and a $O\left((16+\epsilon)^t poly(n)\log n\right)$ deterministic algorithm. 
 	\end{theorem}
 	We conclude our look at computational complexity by noting an interesting open question. 
 	\begin{openquestion}
 		Of particular interest would be the design of a polynomial algorithm for computing the $p$-domination number of a tree.
 	\end{openquestion}
	\section{Conclusion}
 	In summary, this paper has introduced the consideration of partial domination in graphs, which can arise naturally in many applications. Partial domination can be seen as a generalization of classical domination; but as the graph Figure \ref{ex:Disjoint} motivates, partial domination can behave very differently from classical domination.  In particular, we have determined $\gamma_{p}$ for several classes of graphs, developed several bounds on $\gamma_{p}$, and considered the computational complexity of finding $\gamma_{p}$ in a general graph. For some ideas about future direction on this problem, one might consider generalizing concepts and results from domination to partial domination.  In particular there is a famous chain of inequalities first introduced in \cite{Cock1978} related to domination, and it is interesting to consider what meaning and relationships these quantities might have in the setting of partial domination. It would also be interesting to consider what can be said more concretely about partial domination in particular classes of graphs including bounds arising from graph products. Furthermore, it would be interesting to determine the computational complexity for computing $\gamma_{p}$ and constructing $\gamma_{p}$-sets in various classes of graphs. We hope that this introductory paper and promising future directions will promote further interested in considering partial domination.

	\nocite{*}
	\bibliographystyle{plain}
	\bibliography{references}

\end{document}